\documentclass[12pt]{amsart}
\usepackage{a4}
\usepackage{amsthm, amsfonts, amssymb,latexsym}
\usepackage{enumerate, color, comment}
\usepackage[english]{babel}%[english]
\usepackage[latin1]{inputenc}

\newtheorem{theorem}{Theorem}[section]
\newtheorem{cor}[theorem]{Corollary}
\newtheorem{lemma}[theorem]{Lemma}

\newtheorem{conjecture}[theorem]{Conjecture}
\newcommand{\la}{\lambda}
\newcommand{\E}{\mathcal{E}}

\numberwithin{equation}{section}

\newcommand{\beq}[1]{\begin{equation}\label{#1}}
\newcommand{\eeq}{\end{equation}}

\title[On the size of the set $AA+A$]{On the size of the set $AA+A$}

\author[O. Roche-Newton, I. Z. Ruzsa, C. Y. Shen and I. D. Shkredov]{Oliver Roche-Newton, Imre Z. Ruzsa, Chun-Yen Shen and Ilya D. Shkredov}
\address{O. Roche-Newton: Johann Radon Institute for Computational and Applied Mathematics (RICAM), 69 Altenberger Stra{\ss}e, Linz, Austria }
\email{o.rochenewton@gmail.com }
\address{I. Z. Ruzsa: Alfr\'{e}d R\'{e}nyi Institute of Mathematics,
Hungarian Academy of Sciences, Budapest, Hungary }
\email{ruzsa.z.imre@renyi.mta.hu}
\address{Chun-Yen Shen: Department of Mathematics, National Taiwan University and National Center for Theoretical Sciences, Taiwan}
\email{chunyshen@gmail.com}
\address{I.D.~Shkredov: Steklov Mathematical Institute,ul. Gubkina, 8, Moscow, Russia, 119991
and 
IITP RAS, 
Bolshoy Karetny per. 19, Moscow, Russia, 127994 
and 
MIPT, 
Institutskii per. 9, Dolgoprudnii, Russia, 141701}
\email{ilya.shkredov@gmail.com }

\begin{document}

\begin{abstract}
It is established that there exists an absolute constant $c>0$ such that for any finite set $A$ of positive real numbers
$$|AA+A| \gg |A|^{\frac{3}{2}+c}.$$
On the other hand, we give an explicit construction of a finite set $A \subset \mathbb R$ such that $|AA+A|=o(|A|^2)$, disproving a conjecture of Balog.
\end{abstract} 
\maketitle
\section{Introduction}

Given a set $A\subset \mathbb R$, we define its \textit{sum set}, \textit{difference set}, \textit{product set} and \textit{ratio set} respectively as
\begin{align*}
A+A &:=\{a+b:a,b \in A\},
\\ A-A &:= \{a-b: a,b \in A \},
\\ AA &:=\{ab:a,b \in A\},
\\ A/A &:=\{a/b : a,b \in A , b \neq 0 \}.
\end{align*}
A famous conjecture of  Erd\H{o}s and Szemer\'{e}di states that, for all $\epsilon>0$ and for any finite set $A$ of integers, 
$$ \max \{|A+A|,|AA| \} \geq c(\epsilon)|A|^{2-\epsilon},$$
where the constant $c(\epsilon)$ is positive.\footnote{The conjecture is widely believed to be true for finite sets of real numbers, and indeed even complex numbers.} Although this conjecture remains wide-open, several partial results have been attained with gradually improving quantitative bounds. Solymosi \cite{solymosi} notably used a beautiful and elementary geometric argument to prove that, for any finite set $A \subset \mathbb R$,
\begin{equation}
\max \{|A+A|,|AA|\} \gg \frac{|A|^{4/3}}{\log^{1/3}|A|}.
\label{soly}
\end{equation}
Recently, a breakthrough for this problem was achieved by Konyagin and Shkredov \cite{KS}. They adapted and refined the approach of Solymosi, whilst also utilising several other tools from additive combinatorics and discrete geometry, in order to prove that
\begin{equation}
\max \{|A+A|,|AA|\} \gg |A|^{\frac{4}{3}+\frac{1}{20598} -o(1)}.
\label{KS}
\end{equation}
A refinement of the proof of \eqref{KS} by the same authors \cite{KS2} resulted in an improved exponent, and this was improved further in \cite{RSS} to $\frac{4}{3} + \frac{1}{1509}-o(1)$, which stands as the best estimate for the sum-product problem over real numbers at the time of writing. See \cite{KS}, \cite{KS2} and the references contained therein for more background on the sum-product problem.

In this paper, we consider the closely related problem of establishing lower bounds for the set
$$AA+A:=\{ab+c:a,b,c \in A\}.\footnote{
This kind of notation is used with flexibility throughout the paper; for example $AB+CD:=\{ab+cd:a \in A, b \in B, c \in C, d \in D \}$.}$$ It is believed, in the spirit of the Erd\H{o}s-Szemer\'{e}di conjecture, that $AA+A$ is always large. This belief was formalised in the following conjecture of Balog \cite{balog}.

\begin{conjecture}[Balog] \label{conj:balog} For any finite set $A$ of positive real numbers, $|AA+A| \geq |A|^2$.
\end{conjecture}
In the same paper, the following result in that direction was proven:

\begin{theorem} \label{thm:old}
Let $A$ and $B$ be finite sets of positive real numbers. Then
$$|AB+A| \gg |A||B|^{1/2}.$$
In particular,
$$|AA+A| \gg |A|^{3/2}.$$
\end{theorem}

The proof of Theorem \ref{thm:old} uses a similar elementary geometric argument to that of \cite{solymosi}. In fact, one can obtain the same bound by a straightforward application of the Szemer\'{e}di-Trotter Theorem (see \cite[Exercise 8.3.3]{tv}).\footnote{This approach has the advantage that the sign restriction of the sets is removed. Another advantage of this Szemer\'{e}di-Trotter approach is that it has more flexibility in that it can give analogous bounds for sets of the form $AB+C$ where $A,B$ and $C$ are different sets. }

Some progress in this area was made by Shkredov \cite{shkredov}, who built on the approach of Balog in order to prove the following result: 

\begin{theorem} \label{thm:ilya}
For any finite set $A$ of positive real numbers,
\begin{equation}
|A/A+A| \gg  \frac{|A|^{\frac{3}{2}+\frac{1}{82}}}{\log^{\frac{2}{41}}|A|}.
\label{a+A/A}
\end{equation}
\end{theorem}

The first main result of this paper is the following improvement on Theorem \ref{thm:old}: 
\begin{theorem} \label{thm:main} There is an integer $n_0$ and an absolute constant $c>0$ such that for every finite set $A$ of positive reals with $|A| \geq n_0$,
$$|AA+A| \gg |A|^{\frac{3}{2}+c}.$$
\end{theorem}

The constant $c$ is tiny. In the proof we show that we can take $c=2^{-222}$, although in the interests of simplifying the presentation we do not make an attempt to optimise the value of $c$ in the proof. The constant $n_0$ on the other hand is very large. 

%%%%%%%%%%%%%
\begin{comment}

However, the case of the set $A/A+A$ is rather easier to deal with using our methods, and we obtain the following improvement on Theorem \ref{thm:ilya}: 
\begin{theorem} \label{thm:secondary} Let $A$ be a finite set of positive reals. Then
$$|A/A+A| \gg \frac{|A|^{\frac{3}{2}+\frac{1}{26}}}{\log^{5/6}|A|}.$$
\end{theorem}

\end{comment}
%%%%%%%%%%%%%%%

On the other hand, we give a construction which disproves the above conjecture of Balog.

\begin{theorem} \label{thm:construction}
There exists an absolute constant $C$ such that for all $n \in \mathbb N$ there exists $A \subset \mathbb Q$ such that $|A| \geq n$ and 
$$|AA+A| \leq C \frac{|A|^2}{ (\log \log |A|)^{2 \ln 2-1 +o(1)}}.$$
\end{theorem}

% \[ O \left( n^2 (\log \log n )^{1- 2 \log 2 +o(1)} \right) . \]

Although this disproves the full version of Conjecture \ref{conj:balog}, it is plausible that a weaker conjecture that $|AA+A| \gg |A|^{2- \epsilon}$ for all $\epsilon >0$ holds.

We note that the corresponding problem for sets of integers is resolved, up to constant factors, thanks to a nice argument of George Shakan\footnote{See http://mathoverflow.net/questions/168844/sum-and-product-estimate-over-integers-rationals-and-reals.}. Indeed, if $A$ is a finite set of positive integers and $a_{max}$ is the largest element of $A$, then there are no non-trivial solutions to the equation
$$a+a_{max}b=c+a_{max}d$$
such that $a,b,c,d \in A$. This is because such a solution would imply that
$$a_{max}=\frac{c-a}{b-d},$$
but this is a contradiction, since $c-a<a_{max}$ and $b-d \geq 1$ (we may assume without loss of generality that $b>d$). We therefore have
\begin{equation}
|AA+A| \geq |a_{max}A+A| \geq |A|^2.
\label{integers}
\end{equation}
In fact, the only property of the integers used here is that the difference between any two distinct integers has absolute value at least $1$. One can therefore generalise \eqref{integers} to sets of real numbers which are ``well-spaced"; that is, the bound \eqref{integers} holds for any set $A$ of real numbers such that the difference between two distinct elements of $A$ has absolute value at least $1$.

The set $AA+A$ is just one example of a set defined by a combination of additive and multiplicative operations. Such sets have been well studied in recent years; for example,  in \cite{MORNS} and \cite{MORNS2} the dual problem for the set $A(A+A)$ was considered, and it was proven in \cite{MORNS2} that
$$|A(A+A)|\gg |A|^{\frac{3}{2}+\frac{5}{242}-o(1)}.$$
For sets formed from more variables, quantitatively better bounds, in many cases  optimal up to constant and logarithmic factors, have been established. See \cite{MORNS}, \cite{MORNS2} and the references contained therein for more on such variations on the sum-product problem.

\subsection{Notation and preliminary results}

Throughout the paper, the standard notation
$\ll,\gg$ is applied to positive quantities in the usual way. Saying $X\gg Y$ or $Y \ll X$ means that $X\geq cY$, for some absolute constant $c>0$.  The expression $X \approx Y$ means that both $X \gg Y$ and $X \ll Y$ hold. The notation $\lesssim$ and $\gtrsim$ is used to suppress both constant and logarithmic factors. To be precise, the expression $X\gtrsim Y$ or $Y \lesssim X$ means that $X\gg Y/(\log X)^c$, for some absolute constant $c>0$.  All logarithms have base $2$.

Given two finite sets $A,B \subset \mathbb R$, the \textit{additive energy of $A$ and $B$}, denoted $E_+(A,B)$, is the number of solutions to the equation
$$a+b=a'+b'$$
such that $a,a' \in A$ and $b,b' \in B$. The \textit{multiplicative energy of $A$ and $B$} is the number of solutions to the equation $ab=a'b'$ such that $a,a' \in A$ and $b,b' \in B$. Note that $E_+(A,B), E_*(A,B) \leq  \min \{ |A|^2|B|, |A||B|^2, (|A||B|)^{3/2} \}$. The notation $E_+(A,A)$ is shortened to $E_+(A)$, and likewise for multiplicative energy. The following standard lower bounds for additive and multiplicative energy each follow from a simple application of the Cauchy-Schwarz inequality.
\begin{equation}
E_+(A) \geq \frac{|A|^4}{|A + A|} ,\,\,\,\,\,\,\,\,E_+(A) \geq \frac{|A|^4}{|A - A|},\,\,\,\,\,\,\,\,E_*(A) \geq \frac{|A|^4}{|A A|} ,\,\,\,\,\,\,\,\,E_*(A) \geq \frac{|A|^4}{|A /A|}.
\label{CSclassic}
\end{equation}

During the proof of the main theorem we will need to take advantage of a situation in which $A$ has very large additive energy. The classical tool in additive combinatorics for this situation is the Balog-Szemer\'{e}di-Gowers Theorem, which says that if $E_+(A)$ is large then $A$ contains a large subset with small sum set. However, recent progress, particularly in \cite{BW}, \cite{KS2} and \cite{RSS}, has led to the development of different tools which are more effective than the Balog-Szemer\'{e}di-Gowers Theorem in the sum-product setting. In particular we will use the following result, which is Theorem 12 in \cite{RSS}.

\begin{theorem} \label{thm:RSS}
Let $A \subset \mathbb R$. Then there exist $X,Y \subset A$ such that $|X|,|Y| \geq |A|/3$ and
$$E_+(X)^3 \cdot E_*(Y) \lesssim |A|^{11}.$$
\end{theorem}

We will also need the Ruzsa Triangle Inequality, which we state in the following form:
\begin{lemma}[Ruzsa Triangle Inequality]
Let $G$ be an additive group and let $A,B,C \subset G$ be finite sets. Then
$$|A+B||C| \leq |A+C||B+C|.$$
\end{lemma}

During the proof of the main theorem we will need to use some existing sum-product type estimates. The first such result is due to Balog \cite{balog}.
\begin{lemma} \label{balogaa+aa}
Let $A$ and $X$ be finite sets of positive real numbers. Then
$$|AX+AX| \gg |A/A|^{1/2}|X|.$$
\end{lemma}

The second sum-product result that is utilised is a recent result of Shkredov \cite{Sh3}. It is an essential ingredient in the proof of the main theorem in this paper. The key feature of this result is that it gives non-trivial bounds for the case when $B$ is much smaller than $A$ (although the estimates become quantitatively weaker with exponential speed as $B$ becomes smaller). 

\begin{theorem} \label{Shk:assym} There is an absolute constant $C_*$  such that the following statement holds. For any finite sets $A,B \subset \mathbb R$, any $ \alpha \in \mathbb R \setminus \{0\}$ and any $k \geq 1$ such that
$$|B|^{\frac{k}{8}-\frac{1}{4}+\frac{1}{4(k+4)}} \geq |A| C_*^{\frac{k+4}{4}} \log^k(|A||B|)$$
we have
$$|AB|+\frac{|A|^2|B|^2}{E_+(A,\alpha B)} \geq \frac{|A||B|^{\frac{1}{4(k+4)2^k}}}{16}.$$
\end{theorem}

In the above we have simplified the statement slightly to suit the application of the lemma in this paper. For the version of the statement in full generality see \cite[Corollary 33]{Sh3}. We note also here that the application of Theorem \ref{Shk:assym} in the proof of Theorem \ref{thm:main} is the main reason why the constant $c$ in the exponent is so minuscule.

%%%%%%%%%%%%%%%%%%%%
\begin{comment}

For the proof of Theorem \ref{thm:secondary} a direct application of the Szemer\'{e}di-Trotter Theorem is made.

\begin{theorem}[Szemer\'{e}di-Trotter Theorem] \label{thm:SzT}
Let $P$ be a finite set of points in $\mathbb R^2$ and let $L$ be a finite set of lines. Then
$$|\{(p,l)\in P \times L : p \in l\}| \ll (|P||L|)^{2/3}+|P|+|L|.$$
\end{theorem}

Finally, define
$$d(A)=\min_{C \neq \emptyset} \frac{|AC|^2}{|A||C|}.$$
In the proof of Theorem \ref{thm:secondary} we will use the following result, which is \cite[Corollary 8]{KS}. This result comes from an application of the Szemer\'{e}di-Trotter Theorem.
\begin{lemma} \label{thm:ST}
Let $A_1,A_2$ and $A_3$ be finite sets of real numbers and let $\alpha_1,\alpha_2$ and $\alpha_3$ be arbitrary non-zero real numbers. Then the number of solutions to the equation
$$\alpha_1a_1+\alpha_2a_2+\alpha_3a_3=0,$$
such that $a_1 \in A_1$, $a_2 \in A_2$ and $a_3 \in A_3$, is at most 
$$C\cdot d^{1/3}(A_1)|A_1|^{1/3}|A_2|^{2/3}|A_3|^{2/3},$$ for some absolute constant $C$.
\end{lemma}

\end{comment}
%%%%%%%%%%%%%%%%%%%%%

\subsection{Sketch of the proof of Theorem \ref{thm:main}}
The starting point of the proof is the argument used by Balog to prove Theorem \ref{thm:old}. Balog's argument goes roughly as follows:

Consider the point set $A \times A$ in the plane. Cover this point set by lines through the origin. Let us assume for simplicity that all of these lines are equally rich, so we have $|A/A|$ lines with $k$ points on each line. Label the lines $l_1,l_2,\dots,l_{|A/A|}$ in increasing order of steepness. Note that if we take the vector sum of a point on $l_i$ with a point on $l_{i+1}$, we obtain a point which has slope in between those of $l_i$ and $l_{i+1}$. The aim is to show that many elements of $(AA+A) \times (AA+A)$ can be obtained by studying vector sums from neighbouring lines.

Indeed, for any $1 \leq i \leq |A/A|-1$, consider the sum set
$$\{(ab,ac)+(d,e): a\in A, (b,c) \in (A \times A) \cap l_{i}, (d,e) \in (A \times A) \cap l_{i+1} \}.$$
There are at least $|A|$ choices for $(ab,ac)$ and at least $k$ choices for $(d,e)$. Since all of these sums are distinct, we obtain at least $|A|k$ elements of $(AA+A) \times (AA+A)$ lying in between $l_i$ and $l_{i+1}$. Summing over all $1\leq i \leq |A/A|-1$, it follows that
$$|AA+A|^2 \gg |A|^3.$$

There are two rather crude steps in this argument. The first is the observation that there are at least $|A|$ choices for the point $(ab,ac) \in l_i$. In fact, the number of points of this form is equal to the cardinality of product set $AA_{i}$  where $A_{i}$ is the set of $x$-cooordinates of points in $(A \times A) \cap l_i$. The set $AA_i$ could be as small as $|A|$, but one would typically expect it to be considerably larger, in which case we would win.

The second wasteful step comes at the end of the argument, when we only consider sums coming from pairs of lines which are neighbours. This means that we consider only $k-1$ pairs of lines out of a total of ${k \choose 2}$. A crucial ingredient in the proof of \eqref{KS} was the ability to find a way to count sums coming from more than just neighbouring lines. Following that approach, we consider clusters of consecutive lines. It turns out that this approach gives us something better unless there is a strong additive structure between $A$ and $A_i$. To be more precise, we win unless the additive energy $E_+(A, \alpha A_i)$ is maximal for some $\alpha \in \mathbb R \setminus \{0\}$.

We make a small gain unless we are in both of these bad cases. However, if both of these cases occurred it would seem to imply that there is both additive and multiplicative structure existing between the sets $A$ and $A_i$, contradicting the sum-product principle. Indeed, we can use Theorem \ref{Shk:assym} to conclude that at least one of these bad cases does not occur.

However, we can only use the sum-product estimate of Theorem \ref{Shk:assym} if $k$ is larger then $|A|^{\epsilon}$ for some positive $\epsilon$. Therefore, we deal separately with the case when $k$ is very small (and thus the ratio set is very large) at the outset of the proof using a different method. In this case, we make a more straightforward application of the clustering approach from \cite{KS} in order to amplify the bound
\begin{equation}
|AA+AA| \geq |A||A/A|^{1/2},
\label{aa+aaold}
\end{equation}
stated above as Lemma \ref{balogaa+aa}. This approach gives an improvement on \eqref{aa+aaold} unless $A$ is very additively structured. If we have something better than \eqref{aa+aaold} then we can use the Ruzsa Triangle Inequality to complete the proof. If not then we can use the additive structure of $A$ to invoke sum-product type estimates, namely Theorem \ref{thm:RSS}, to complete the proof.

%%%%%%%%%%%%%%%
\begin{comment}
\subsection{Sketch of the proof of Theorem \ref{thm:secondary}} We follow the same setup as the proof of Theorem \ref{thm:main} sketched above. However, the useful fact that $(A/A)^{-1}=A/A$ allows us to use a sum-product type theorem of Elekes, Nathanson and Ruzsa \cite{ENR} to rule out the case when $k$ is small. We therefore do not need to use the assymmetric sum-product estimate from \cite{Sh3}, and instead can use more direct applications of the Szemer\'{e}di-Trotter Theorem which give good quantitative bounds as $k$ gets larger. We can then effectively combine the clustering approach with some ideas from \cite{shkredov} in order to make the two inefficient steps mentioned above less wasteful.
\end{comment}
%%%%%%%%%%%%%%%

\section{Proof of main theorem}

\subsection{A lower bound for $|AA+AA|$}

The following lemma, which may be of independent interest, helps us to deal with the case when the ratio set is very large. The result gives an improvement on Lemma \ref{balogaa+aa} unless the additive energy of $X$ is very large. Its proof also gives a preview of the clustering strategy which is the fulcrum of the proof of the main theorem. A similar argument with different bounds can be found in the proof of Theorem 2 in \cite{IRNR}. 

\begin{lemma} \label{bigratio}
Let $A$ and $X$ be finite sets of positive reals and write $E_+(X)=|X|^3 /K$. Then
$$|AX+AX| \gg |X| |A/A|^{1/2} K^{1/8}.$$
\end{lemma}

\begin{proof}

Following the notation of \cite{KS}, for a real nonzero $\lambda$, define
$$ \mathcal A_{\lambda}:= \left \{(x,y) \in A \times A : \frac{y}{x}=\lambda \right \},$$
and its projection onto the horizontal axis,
$$A_{\lambda}:=\{x:(x,y) \in \mathcal A_{\lambda}\}.$$
Note that $|\mathcal A_{\la}|=|A_{\la}|=|A \cap \la A|$ and
\begin{equation}
\sum_{\la} |A_{\la}|=|A|^2.
\label{obvious}
\end{equation}

Observe that $A \times A$ is covered by $|A/A|$ lines through the origin, and indeed $A/A$ is precisely equal to the set of slopes of these lines. Label the lines $l_1,l_2, \dots, l_{|A/A|}$ in increasing order of steepness, so the line $l_i$ has gradient strictly less than that of $l_j$ if and only if $i<j$.

For each $\lambda \in A/A$, we identify an arbitrary element from $\mathcal A_{\la}$, which we label $(a_{\lambda},\lambda a_{\lambda})$. Then, fixing two distinct slopes $\lambda_1$ and $\lambda_2$ from $A/A$ and following the observation of Balog \cite{balog}, we note that at least $|X|^2$ distinct elements of $(AX+AX) \times (AX+AX)$ are obtained by summing pairs of vectors from the two lines with slope $\lambda_1$ and $\lambda_2$. Indeed,
$$\Delta(X)\cdot(a_{\lambda_1},\lambda_1 a_{\lambda_1})  + \Delta(X) \cdot (a_{\lambda_2},\lambda_2a_{\lambda_2})  \subset (AX+AX) \times (AX+AX),$$
where
$$
\Delta(X)=\{(x,x) : a \in X \}.
$$
Note that these $|X|^2$ vector sums have slope in between $\lambda_1$ and $\lambda_2$. This is a consequence of the observation of Solymosi \cite{solymosi} that the sum set of $m$ points on one line through the origin and $n$ points on another line through the origin consists of $mn$ points lying in between the two lines.  This fact expresses linear independence of two vectors in the two given directions, combined with the fact that multiplication by positive numbers preserves order of reals.\footnote{It is worth noting here that this fact is dependent on the points lying inside the positive quadrant of the plane, which is why the assumption that $A$ consists of strictly positive reals is needed for this proof.}

Following the strategy of \cite{KS}, we split the family of $|A/A|$ slopes into clusters of $M$ consecutive lines, where $2\leq M \leq |A/A|$ is a parameter to be specified later. For example, the first cluster is $U_1= \{l_1,\dots,l_{M}\}$, the second is $U_2=\{l_{M+1},\dots,l_{2M}\}$, and so on.. The idea is to show that each cluster determines many different elements of $(AX+AX) \times (AX+AX)$. Since the slopes of these elements are in between the maximal and minimal values in that cluster, we can then sum over all clusters without overcounting.

If a cluster contains exactly $M$ lines, then it is called a \textit{full cluster}. Note that there are $\left\lfloor \frac{|A/A|}{M} \right\rfloor \geq \frac{|A/A|}{2M}$ full clusters, since we place exactly $M$ lines in each cluster, with the possible exception of the last cluster which contains at most $M$ lines.

Let $U$ be a full cluster. The forthcoming analysis will work in exactly the same way for any full cluster, and so for simplicity of notation we deal only with the first cluster $U=\{l_1,\dots, l_M\}$. We will sometimes abuse notation by identifying $U$ with the slopes of the lines in $U$.  Let $\mu$ denote the number of elements of $(AX+AX) \times (AX+AX)$ which lie in between the slopes of $l_1$ and $l_M$. Then,

%Let $k=|\mathcal{L}_j|$ and write $\mathcal L_j=\{l_1,l_2,\dots,l_{k}\}$, where the slope of $l_i$ is smaller than that of $l_{i+1}$. We denote the slope of the line %$l_i$ by $\lambda_{l_i}$. The set of lines $\mathcal L_j$ is now divided into clusters of size $M$, with each cluster consisting of consecutive lines from $\mathcal %L_j$. So, the first cluster is $U_1=\{l_1,l_2,\dots,l_M\}$, and a cluster $U_i$ consists of $M$ consecutive lines, with the possible exception of the last cluster. %However, there are at least $\lfloor \frac{k}{M} \rfloor \geq \frac{k}{2M}$ clusters containing $M$ lines.

%Fix a cluster of lines $U_i=\{l_{(i-1)M+1},\dots l_{iM}\}$ of size $M$. The aim is to show that by taking vector sums of pairs of lines from $U_i$, we obtain many %distinct elements of $(AA+A) \times (AA+A)$, all of which have a slope which is in between $\lambda_{l_{(i-1)M+1}}$ and $\lambda_{l_{iM}}$. We will then sum %over all clusters.

\begin{equation}\mu \geq |X|^2 {M \choose 2} - \sum_{\lambda_1,\lambda_2,\lambda_3,\lambda_4 \in U: \{\lambda_1,\lambda_2\} \neq \{\lambda_3,\lambda_4\}} \E(\lambda_1,\lambda_2,\lambda_3,\lambda_4),
\label{mucount2}
\end{equation}
where
\begin{align*}
\E(\lambda_1,\lambda_2,\lambda_3,\lambda_4)  := &
 |\{[\Delta(X) \cdot (a_{\la_1},\la_1 a_{\la_1}) +\Delta(X) \cdot (a_{\lambda_2},\lambda_2a_{\lambda_2})]
\\& \cap [\Delta(X) \cdot (a_{\lambda_3},\lambda_3a_{\lambda_3})+\Delta(X) \cdot (a_{\lambda_4},\lambda_4a_{\lambda_4})] \}|.
\end{align*}
In \eqref{mucount2}, the first term is obtained by counting sums from all pairs of distinct lines from $U$. The second error term covers the overcounting of elements that are counted more than once in the first term. 

The next task is to obtain an upper bound for $\E(\lambda_1,\lambda_2,\lambda_3,\lambda_4)$ for an arbitrary quadruple $(\la_1,\la_2,\la_3,\la_4)$ which satisfies the aforementioned conditions.

Suppose that
\begin{align*}
z=(z_1,z_2) \in  & [ \Delta(X) \cdot (a_{\lambda_1},\lambda_1a_{\lambda_1})+\Delta(X) \cdot (a_{\lambda_2},\lambda_2a_{\lambda_2})]
\\ & \cap  [\Delta(X) \cdot (a_{\lambda_3},\lambda_3a_{\lambda_3})+\Delta(X) \cdot (a_{\lambda_4},\lambda_4a_{\lambda_4})],
\end{align*}
that is
$$
(z_1,z_2) =(x_1a_{\la_1},x_1\lambda_1a_{\la_1})+(x_2a_{\lambda_2},x_2\lambda_2a_{\lambda_2})
=(x_3a_{\la_3},x_3\lambda_3a_{\la_3})+(x_4a_{\lambda_4},x_4\lambda_4a_{\lambda_4}),
$$
for some $x_1,x_2,x_3,x_4 \in X$. Therefore,
$$\begin{array}{lccccccc}
z_1&=&x_1a_{\la_1}+x_2a_{\lambda_2}&=&x_3a_{\la_3}+x_4a_{\lambda_4}\\
z_2&=&x_1\lambda_1a_{\la_1}+x_2\lambda_2a_{\lambda_2}&=&x_3\lambda_3a_{\la_3}+x_4\lambda_4a_{\lambda_4}.
\end{array}$$
It follows from the conditions on the quadruple $(\la_1,\la_2,\la_3,\la_4)$ that at least one of its members differs from the other three. Without loss of generality $\la_4\neq \la_1,\la_2,\la_3$. Then
$$0=x_1\lambda_1a_{\la_1}+x_2\lambda_2a_{\lambda_2}-x_3\lambda_3a_{\la_3}-x_4\lambda_4a_{\lambda_4} - \lambda_4(x_1a_{\la_1}+x_2a_{\lambda_2}-x_3a_{\la_3}-x_4a_{\lambda_4}),$$
and thus
\begin{equation}
0=x_1a_{\la_1}(\lambda_1-\lambda_4)+x_2a_{\lambda_2}(\lambda_2-\lambda_4)+x_3a_{\la_3}(\lambda_4-\lambda_3).
\label{STsetuppp}
\end{equation}

Note that the values $a_{\la_1}(\lambda_1-\lambda_4), a_{\lambda_2}(\lambda_2-\lambda_4)$ and $a_{\la_3}(\lambda_4-\lambda_3)$ are all non-zero. We have shown that each contribution to $\E (\lambda_1,\lambda_2,\lambda_3,\lambda_4)$ determines a solution to \eqref{STsetuppp}. Furthermore, the solution $(x_1,x_2,x_3)$ to \eqref{STsetuppp} that we obtain via this deduction is unique. That is, if we start out with a different element
$$z \in [\Delta(X)(a_{\lambda_1},\lambda_1a_{\lambda_1})+\Delta(X)(a_{\lambda_2},\lambda_2a_{\lambda_2})]
\cap  [\Delta(X)(a_{\lambda_3},\lambda_3a_{\lambda_3})+\Delta(X)(a_{\lambda_4},\lambda_4a_{\lambda_4})],$$
we obtain a different solution to \eqref{STsetuppp}. 

Therefore, in order to obtain an upper bound for $\E(\la_1,\la_2,\la_3, \la_4)$ it will suffice to obtain an upper bound for the number of solutions to \eqref{STsetuppp}.  Let $T$ denote the number of such solutions. Then, by two applications of the Cauchy-Schwarz inequality, we have
 \begin{align*}
T&= \sum _{x_1 \in X} \left | \left\{(x_2,x_3) \in X \times X : x_1=\frac{a_{\lambda_2}(\lambda_2-\lambda_4)}{a_{\la_1}(\la_4-\la_1)}x_2+\frac{a_{\la_3}(\lambda_4-\lambda_3)}{a_{\la_1}(\la_4-\la_1)}x_3 \right \} \right |
\\& \leq |X|^{1/2} E_+^{1/2}\left (\frac{a_{\lambda_2}(\lambda_2-\lambda_4)}{a_{\la_1}(\la_4-\la_1)}X, \frac{a_{\la_3}(\lambda_4-\lambda_3)}{a_{\la_1}(\la_4-\la_1)}X \right)
\\&\leq |X|^{1/2}E_+(X)^{1/2}.
\end{align*}

Therefore, by \eqref{mucount2}, we have
\begin{equation}
 \mu \geq \frac{ |X|^2 M^2 }{4} - M^4 |X|^{1/2} E_+(X)^{1/2}.
\label{mucount3}
\end{equation}
We now choose the integer parameter $2 \leq M \leq |A/A|$. We want $M$ to be as large as possible such that the first term in \eqref{mucount3} is dominant. A suitable choice is
\begin{equation}
M:= \left \lfloor \left( \frac{ |X|^{3/2}}{8E_+^{1/2}(X) } \right )^{1/2} \right \rfloor .
\label{Mchoice22}
\end{equation}
With this choice we need to verify the condition that $2 \leq M \leq |A/A| $. It is not difficult to check that if these conditions are violated then the claimed result holds.

Indeed, if $M < 2$ then $E_+(X) \gg |X|^3$. Then $K=|X|^3 / E_+(X)$ is an absolute constant, and the result holds by Lemma \ref{balogaa+aa}. On the other hand, if $M > |A/A|$ then it follows that $K^{1/4} > |A/A|$. If this happens then we can use \eqref{CSclassic} to complete the proof since
$$|AX+AX| \geq |X+X| \geq \frac{|X|^4}{E_+(X)} = |X|K > |X| |A/A|^{1/2} K^{7/8} \geq |X| |A/A|^{1/2} K^{1/8}.$$

Therefore, we can assume that the condition $2 \leq M \leq |A/A|$ is satisfied, and then by \eqref{Mchoice22} we have $M \approx K^{1/4}.$  We now have the bound $\mu \gg |X|^2 M^2$. Summing over all full clusters, of which there are at least $|A/A|/2M$, we have
$$|AX+AX|^2 \gg |X|^2|A/A| M \gg |X|^2|A/A| K^{1/4}.$$
This completes the proof of the lemma.

\end{proof}

\subsection{The case when the ratio set is large}

Let $\epsilon > 0$ be a small fixed positive constant. We do not make an effort to optimise the choice of $\epsilon$, and instead we simply fix $\epsilon=1/600$ for ease of calculation. In this subsection we assume that $|A/A| \geq |A|^{2-1/24+\epsilon}$.

We make an application of Theorem \ref{thm:RSS}. Let $X \subset A$ be the set given by this and write $E_+(X)=|X|^3/K$.  Recall also that $|X| \gg |A|$.

\textbf{Subcase 1} - Suppose that $K \geq |A|^{1/6 - \epsilon}$. Then by Lemma \ref{bigratio} and the Ruzsa Triangle Inequality we have 
$$|A|^{3+3\epsilon/8} \leq |A|^{3+\epsilon/2-\epsilon/8} \ll |X| |A/A|^{1/2} K^{1/8}|A| \ll |AX+AX||A| \leq |AA+AA| |A| \leq |AA+A|^2.$$
This implies the required bound $|AA+A| \gg |A|^{3/2+c}$.

\textbf{Subcase 2} - Suppose that $K \leq |A|^{1/6 - \epsilon}$. Then 
by Theorem \ref{thm:RSS}, we have
$$|A|^{9-1/2+3\epsilon} \cdot E_*(Y) \ll \frac{|X|^9}{K^3} \cdot E_*(Y) = E_+(X)^3 \cdot E_*(Y) \lesssim  |A|^{11}  .$$
Simplifying this inequality, applying \eqref{CSclassic} and using the bound $|Y| \gg |A|$ (given by Theorem \ref{thm:RSS}) yields
$$\frac{|A|^4}{|YY|} \ll \frac{|Y|^4}{|YY|} \leq E_*(Y) \lesssim  |A|^{2+1/2-3 \epsilon}.$$
It then follows that
$$|AA+A| \geq |AA| \geq |YY| \gtrsim  |A|^{3/2 + 3\epsilon} \geq |A|^{3/2 +\epsilon}.$$
The last inequality above follows by taking $n_0$ sufficiently large.

In both of these two subcases we achieve the desired result. Therefore, for the rest of the proof, we may assume that $|A/A| \leq |A|^{2-1/24+\epsilon}$. Recalling the choice $\epsilon=1/600$, we thus have that $|A/A| \leq |A|^{2-1/25}$.

\subsection{Some initial dyadic pigeonholing}

Consider the point set $A \times A$ in the plane. At the outset, we perform a dyadic decomposition, and then apply the pigeonhole principle, in order to find a large subset of $A\times A$ consisting of points lying on lines through the origin which contain between $\tau$ and $2\tau$ points, where $\tau$ is some real number.

Let $S_{\tau}$ be defined by
$$S_{\tau}:=\{\lambda: \tau \leq |A \cap \lambda A| < 2\tau \}.$$
After dyadically decomposing the sum \eqref{obvious}, we have
$$|A|^2=\sum_{\la} |A_{\la}| = \sum_{j=1}^{\lceil \log|A| \rceil} \sum_{\la \in S_{2^{j-1}}}|A_{\la}| .$$
Applying the pigeonhole principle, we deduce that there is some $\tau$ such that
\begin{equation}
\sum_{\la \in S_{\tau}}|A_{\la}| \geq \frac{|A|^2}{\lceil \log|A| \rceil} \geq \frac{|A|^2}{2\log|A|}.
\label{sumboundd}
\end{equation}
\begin{comment}
Since $\tau \leq |A|$, this implies that
\begin{equation}
|S_{\tau}| \geq \frac{|A|}{2\log |A|}.
\label{Sboundd}
\end{equation}
\end{comment}
Also, since $|A_{\la}| < 2\tau$ for any $\la \in S_{\tau}$, we have
\begin{equation}
\tau|S_{\tau}| \gg \frac{|A|^2}{ \log|A|}.
\label{tauboundd}
\end{equation}

We then make another dyadic decomposition to control the sizes of the product sets $AA_{\la}$. Define $S_{\tau}^{(t)} \subseteq S_{\tau}$ to be the set
$$ S_{\tau}^{(t)}:= \{ \la \in S_{\tau} : t |A| \leq |AA_{\la}| \leq 2t |A| \}. $$
We have
$$|S_{\tau}|= \sum_{t=1}^{\lceil \log|A| \rceil} |S_{\tau}^{(t)}|$$
and so there is some $t_0$ such that $|S_{\tau}^{t_0}| \gg |S_{\tau}| / \log |A|$. We shorten $S_{\tau}^{t_0}$ to $S$, so we have
$$S:=\{ \la \in S_{\tau} : |A|t_0 \leq |AA_{\la}| < 2|A|t_0 \}$$
and
$$|S| \gg \frac{|S_{\tau}|}{\log |A|}.$$
Therefore, by \eqref{tauboundd} we have
\begin{equation}
\tau |S| \gg \frac{|A|^2}{\log ^2|A|}.
\label{Staubound}
\end{equation}
Note also, because of our assumption that the ratio set is not large, we have 
$$|S| \leq |A/A| \leq |A|^{2-\frac{1}{25}}$$ 
and so it follows from \eqref{Staubound} that
\begin{equation}
\tau \gg \frac { |A|^{\frac{1}{25}} }{\log^2|A|}.
\label{taubound2}
\end{equation}

\subsection{Application of the assymetric sum-product estimate}

We would like to apply Theorem \ref{Shk:assym} with $B=A_{\la}$ and $\la \in S$. To do this we need to choose a sufficiently large value of $k$ so that the inequality
$$|A_{\la}|^{\frac{k}{8}-\frac{1}{4}+\frac{1}{4(k+4)}} \geq |A| C_*^{\frac{k+4}{4}} \log^k(|A||A_{\la}|)$$
is satisfied. Because of the bound \eqref{taubound2}, it suffices to choose $k$ such that
$$\frac{|A|^{\frac{(k-2)(k+4)+2}{25\cdot 8 (k+4)}}}{\log^k|A|} \geq |A| C_*'^{\frac{k+4}{4}}2^k \log^k|A|,$$
which rearranges to give
\begin{equation}
|A|^{\frac{(k-2)(k+4)+2 - 200(k+4)}{200 (k+4)}} \geq  C_*''^{\frac{k+4}{4}} \log^{2k}|A|.
\label{uglyneed}
\end{equation}
By choosing the value $n_0$ in the statement of the theorem to be sufficiently large\footnote{This is the only time in the proof where we use the largeness of $n_0$ in a meaningful way. In all other instances we are just using this assumption to superficially remove log factors.}, it will suffice that exponent on the left hand side of \eqref{uglyneed} is larger than some small positive value. For example, it will suffice to choose $k$ sufficiently large such that
$$\frac{(k-2)(k+4)+2 - 200(k+4)}{200 (k+4)} \geq \frac{1}{200}.$$
One can then directly verify that $k=203$ is sufficient.

Therefore, we can apply Theorem \ref{Shk:assym} with this choice of $k$ to deduce that, for any non-zero $\alpha$,
$$|AA_{\la}| +\frac{|A|^2|A_{\la}|^2}{ E_+(A, \alpha A_{\la})} \geq \frac{|A||A_{\la}|^{\frac{1}{4(k+4)2^k}}}{16} \geq \frac{|A|\tau^{\frac{1}{2^{213}}}}{16}. $$
In particular, we have either 
\begin{equation}
|AA_{\la}| \geq \frac{1}{32}|A|\tau^{2^{-213}}
\label{bigprod}
\end{equation}
or
\begin{equation}
E_+(A,\alpha A_{\la}) \leq 32 |A||A_{\la}|^2 \tau^{-2^{-213}} \leq 128 |A| \tau^{2-2^{-213}},
\label{smallenergy}
\end{equation}
for all non-zero $\alpha$.

\subsection{The case when $t_0$ is ``large"}

In this section we will show that in the situation of \eqref{bigprod} we are done, and that it can thus be assumed that \eqref{smallenergy} holds.

Suppose that 
$$t_0 \geq \frac{\tau^{\frac{1}{2^{213}}}}{64}$$
and so
$$|AA_{\la}| \geq \frac{ |A|\tau^{\frac{1}{2^{213}}}}{64}$$ 
for all $\la \in S$.  In this case we can obtain the desired result with a small modification of Balog's argument in \cite{balog} (see also \cite{shkredov}).

After carrying out the aforementioned pigeonholing argument, we have a set of $|S|$ lines through the origin, each containing approximately $\tau$ points from $A \times A$. Label the slopes of these lines as $\la_1,\la_2,\dots,\la_{|S|}$ in increasing order of size. For any $1 \leq i \leq |S|-1$, consider the sum set
\begin{equation}
\mathcal A_{\la_i}+\mathcal A_{\la_{i+1}} \cdot \Delta(A) \subset (A+AA) \times (A+AA),
\label{vectorsumm}
\end{equation}
where again $\Delta(A)=\{(a,a):a \in A\}$. Note that $\mathcal A_{\la_{i+1}} \cdot \Delta(A)$ has cardinality $|A_{\la_{i+1}}A|$, and therefore the set $\mathcal A_{\la_i}+\mathcal A_{\la_{i+1}} \cdot \Delta(A)$ has $|A_{\la_i}||A_{\la_{i+1}}A|$ elements, all of which lie in between the lines through the origin with slopes $\la_i$ and $\la_{i+1}$. This is a consequence of the same observation of Solymosi that was used in the proof of Lemma \ref{bigratio}.

Summing over all $1 \leq i < |S|$, and applying \eqref{Staubound} and \eqref{taubound2}, we have
\[
|AA+A|^2 \geq \sum_{i=1}^{|S|-1} |A_{\la_i}||AA_{\la_{i+1}}| \gg |S| \tau |A| \tau^{\frac{1}{2^{213}}} \gg  \frac{|A|^3 \tau^{\frac{1}{2^{213}}} }{\log^2|A|} \gg \frac{ |A|^{3+2^{-218}} }{\log^3|A|} \geq |A|^{3+2^{-219}} .
\]
The final inequality above follows by choosing the value $n_0$ in the statement of the theorem sufficiently large.
This implies the required bound. 

Therefore we may assume henceforth that
$$t_0 < \frac{\tau^{\frac{1}{2^{213}}}}{64}$$
and thus
$$|AA_{\la}| < \frac{ |A|\tau^{\frac{1}{2^{213}}}}{32}$$
for all $\la \in S$. 

Returning to the previous subsection, we now have that \eqref{bigprod} does not hold, and so \eqref{smallenergy} does hold. That is, we now assume for the remainder of the proof that
\begin{equation}
E_+(A,\alpha A_{\la}) \leq 128 |A| \tau^{2-2^{-213}},
\label{keyenergy}
\end{equation}
for all $\la \in S$ and any $\alpha \neq 0$. This non-trivial lower bound for the additive energy of $A$ and $\alpha A_{\la}$ will play a vital role later in the proof.

\subsection{Clustering setup}

Here we repeat the clustering setup used in the proof of Lemma \ref{bigratio}, but with some modifications in order to work towards a lower bound for $|AA+A|$. Label the lines $l_1,l_2, \dots, l_{|S|}$, corresponding to the slopes of $S$, in increasing order of steepness. So the line $l_i$ has gradient strictly less than that of $l_j$ if and only if $i<j$.

For each $\lambda \in S$, we identify an element from $\mathcal A_{\la}$, which we label $(a_{\lambda},\lambda a_{\lambda})$. These fixed points are chosen completely arbitrarily.

Then, fixing two distinct slopes $\lambda$ and $\lambda'$ from $S$ and following the observation of Balog \cite{balog}, we note that at least $\tau|A|$ distinct elements of $(AA+A) \times (AA+A)$ are obtained by summing points from the two lines. Indeed,
$$\mathcal A_{\lambda}+(a_{\lambda'},\lambda'a_{\lambda'}) \cdot \Delta(A) \subset (AA+A) \times (AA+A).$$ 
Once again, these vector sums are all distinct and have slope in between $\lambda$ and $\lambda'$.

Following the strategy of Konyagin and Shkredov \cite{KS}, we split the family of $|S|$ slopes into clusters of $2M$ consecutive slopes, where $2\leq 2M \leq |S|$ and $M$ is a parameter to be specified later. We then split each cluster arbitrarily into two disjoint subclusters of size $M$. For example, we have $U_1=V_1 \sqcup W_1$ where $V_1=\{l_1,\dots,l_M\}$ and $W_1=\{l_{M+1},\dots,l_{2M}\}$. The idea is to show that each cluster determines many different elements of $(A+AA) \times (A+AA)$.

If a cluster contains exactly $2M$ lines, then it is called a \textit{full cluster}. Note that there are $\left\lfloor \frac{|S_{\tau}|}{2M} \right\rfloor \geq \frac{|S_{\tau}|}{4M}$ full clusters, since we place exactly $2M$ lines in each cluster, with the possible exception of the last cluster which contains at most $2M$ lines.

The proceeding analysis will work in exactly the same way for any full cluster, and so for simplicity of notation we deal only with the first cluster $U_1$. We further simplify this by writing $U_1=U$, $V_1=V$ and $W_1=W$.

Let $\mu$ denote the number of elements of $(AA+A) \times (AA+A)$ which lie in between $l_1$ and $l_{2M}$. Then, similarly to the proof of Lemma \ref{bigratio},

%Let $k=|\mathcal{L}_j|$ and write $\mathcal L_j=\{l_1,l_2,\dots,l_{k}\}$, where the slope of $l_i$ is smaller than that of $l_{i+1}$. We denote the slope of the line %$l_i$ by $\lambda_{l_i}$. The set of lines $\mathcal L_j$ is now divided into clusters of size $M$, with each cluster consisting of consecutive lines from $\mathcal %L_j$. So, the first cluster is $U_1=\{l_1,l_2,\dots,l_M\}$, and a cluster $U_i$ consists of $M$ consecutive lines, with the possible exception of the last cluster. %However, there are at least $\lfloor \frac{k}{M} \rfloor \geq \frac{k}{2M}$ clusters containing $M$ lines.

%Fix a cluster of lines $U_i=\{l_{(i-1)M+1},\dots l_{iM}\}$ of size $M$. The aim is to show that by taking vector sums of pairs of lines from $U_i$, we obtain many %distinct elements of $(AA+A) \times (AA+A)$, all of which have a slope which is in between $\lambda_{l_{(i-1)M+1}}$ and $\lambda_{l_{iM}}$. We will then sum %over all clusters.

\begin{equation}\mu \geq \tau|A| M^2 - \sum_{\lambda_1, \lambda_3 \in V ,\lambda_2,\lambda_4 \in W: \{\lambda_1,\lambda_2\} \neq \{\lambda_3,\lambda_4\}} \mathcal E(\lambda_1,\lambda_2,\lambda_3,\lambda_4),
\label{mucount22}
\end{equation}
where
$$\mathcal E(\lambda_1,\lambda_2,\lambda_3,\lambda_4):=|\{z \in (\mathcal A_{\lambda_1}+(a_{\lambda_2},\lambda_2a_{\lambda_2})\cdot\Delta(A))\cap (\mathcal A_{\lambda_3}+(a_{\lambda_4},\lambda_4a_{\lambda_4})\cdot\Delta(A)) \}|.$$

The next task is to obtain an upper bound for $\mathcal E(\lambda_1,\lambda_2,\lambda_3,\lambda_4)$ for an arbitrary quadruple $(\la_1,\la_2,\la_3,\la_4)$ which satisfies the aforementioned conditions.

\subsection{Bounding $\mathcal E(\lambda_1,\lambda_2,\lambda_3,\lambda_4)$ in the case when $\la_4 \neq \la_2$} Let us fix $\la_1,\la_2,\la_3, \la_4$ with $\la_4 \neq \la_2$.
%\footnote{There is a subtle but important difference here between this proof and the similar proof of Lemma \ref{bigratio}. In this proof we make a clear distinction between the case when $\la_2=\la_4$ and the case when $\la_2 \neq \la_4$, whereas in the proof of Lemma \ref{bigratio} (and similarly in the proof of the main result in \cite{KS}) it is assumed without loss of generality that $\la_4 \neq \la_2$. The reason that we cannot do this here is that the sizes of the sets $A$ and $A_\la$ that will be dealt with in the forthcoming analysis are rather different. We eliminate one variable before carrying out this analysis, and depending on which variable is eliminated the analysis may become rather different. In fact, this is rather a significant technical obstruction in the proof, since the case $\la_2=\la_4$ becomes difficult when $\tau$ is smaller than $|A|^{1/2}$ as the usual applications of Szemer\'{e}di-Trotter become ineffective. See the proof of Theorem \ref{thm:secondary}, which uses such techniques to obtain a significantly larger exponent of the form $3/2+c$. The arguments there for the case when $\la_4 \neq \la_2$ work equally well here.} 
Note that this assumption implies that $\la_4\neq \la_1,\la_2,\la_3$. 

Suppose that 
$$z=(z_1,z_2) \in [\mathcal A_{\lambda_1}+(a_{\lambda_2},\lambda_2a_{\lambda_2})\cdot\Delta(A)]\cap [\mathcal A_{\lambda_3}+(a_{\lambda_4},\lambda_4a_{\lambda_4})\cdot\Delta(A)].$$
Then
\begin{align*}
(z_1,z_2) &=(a_1,\lambda_1a_1)+(a_{\lambda_2}a,\lambda_2a_{\lambda_2}a)
\\&=(a_3,\lambda_3a_3)+(a_{\lambda_4}b,\lambda_4a_{\lambda_4}b),
\end{align*}
for some $a_1 \in A_{\lambda_1}$, $a_3 \in A_{\lambda_3}$ and $a,b \in A$. Therefore,
\begin{align*}
z_1&=a_1+a_{\lambda_2}a=a_3+a_{\lambda_4}b
\\z_2&=\lambda_1a_1+\lambda_2a_{\lambda_2}a=\lambda_3a_3+\lambda_4a_{\lambda_4}b.
\end{align*}
We have
$$0=\lambda_1a_1+\lambda_2a_{\lambda_2}a-\lambda_3a_3-\lambda_4a_{\lambda_4}b - \lambda_4(a_1+a_{\lambda_2}a-a_3-a_{\lambda_4}b),$$
and thus
\begin{equation}
0=a_{\lambda_2}(\lambda_2-\lambda_4)a+(\lambda_1-\lambda_4)a_1+(\lambda_4-\lambda_3)a_3.
\label{Cauchysetup}
\end{equation}

Note that the values $\lambda_1-\lambda_4, a_{\lambda_2}(\lambda_2-\lambda_4)$ and $\lambda_4-\lambda_3$ are all non-zero (recall here that the value $a_{\la_2}$ is a fixed constant given by our earlier choice of a fixed point on each line). Let $T_1$ denote the number of solutions to \eqref{Cauchysetup} such that $(a,a_1,a_3) \in A \times A_{\la_1} \times A_{\la_3}$. We have shown that each contribution to $\mathcal E (\lambda_1,\lambda_2,\lambda_3,\lambda_4)$ determines a solution to \eqref{Cauchysetup}. Furthermore, the solution to \eqref{Cauchysetup} that we obtain via this deduction is unique, and so
$$\mathcal E(\la_1,\la_2,\la_3,\la_4) \leq T_1.$$

By the Cauchy-Schwarz inequality, we have
\begin{align*}
T_1&=\sum_{a_1 \in A_{\la_1}} \left | \left \{(a,a_3) \in A \times A_{\la_3} : a_1=\frac{a_{\la_2}(\la_2-\la_4)}{\la_4-\la_1}a+ \frac{\la_4-\la_3}{\la_4-\la_1}a_3  \right \} \right |
\\& \leq |A_{\la_1}|^{1/2} \left( \sum_{x}\left | \left \{(a,a_3) \in A \times A_{\la_3} : x=\frac{a_{\la_2}(\la_2-\la_4)}{\la_4-\la_1}a+ \frac{\la_4-\la_3}{\la_4-\la_1}a_3  \right \} \right |^2 \right)^{1/2}
\\&=|A_{\la_1}|^{1/2} E^+\left (A, \frac{\la_4-\la_3}{a_{\la_2}(\la_2-\la_4)} A_{\la_1} \right)^{1/2}.
\end{align*}
By inequality \eqref{keyenergy} and the trivial bound $|A_{\la_1}| \leq |A|$, we then have
$$\mathcal E(\lambda_1,\lambda_2,\lambda_3,\lambda_4) \leq 12 |A|\tau^{1-2^{-214}}.$$
Therefore,
\begin{equation}
\mu \geq M^2 |A|\tau -12 M^4 |A|\tau^{1-2^{-214}}  -\sum_{\lambda_1, \lambda_3 \in V ,\lambda_2 \in W: \lambda_1 \neq \lambda_3} \mathcal E(\lambda_1,\lambda_2,\lambda_3,\lambda_2).
\label{mu}
\end{equation}

We now impose a condition on the parameter $M$ (recall that we will choose an optimal value of $M$ at the conclusion of the proof) to ensure that the first error term is dominated by the main term. We need
$$12 M^4 |A|\tau^{1-2^{-214}}  \leq \frac{M^2|A|\tau}{2},$$
which simplifies to 
\begin{equation} \label{Mcondd}
M \leq  \frac{ \tau^{2^{-215}} } {\sqrt {24}} .
\end{equation}
With this restriction on $M$, we now have
\begin{equation}
\mu \geq \frac{M^2 |A|\tau}{2}  -\sum_{\lambda_1, \lambda_3 \in V ,\lambda_2 \in W: \lambda_1 \neq \lambda_3} \mathcal E(\lambda_1,\lambda_2,\lambda_3,\lambda_2).
\label{muu}
\end{equation}
It remains to bound the error term in \eqref{muu}.

\subsection{Bounding $\mathcal E(\lambda_1,\lambda_2,\lambda_3,\lambda_4)$ in the case $\la_4 = \la_2$}
%Recall that we need $2 \leq M \leq |S_{\tau}|$. It is easy to check that the upper bound for $M$ is satisfied. Indeed, it follows from \eqref{Sbound} that
%$$M \leq |A|^{1/3} \leq \frac{|A|}{4 \log |A|} \leq  |S_{\tau}|,$$
%where the second inequality is true for sufficiently large $|A|$. Since smaller sets can be dealt with by choosing sufficiently small implied constants in the statement, we may assume that $M \leq |S_{\tau}|$.

%Let us first assume that $2 \leq M$. We will deal with the case when $M$ lies outside of this range at the conclusion of the proof. It follows from the choice of $M$ that

Now we fix $\la_1,\la_2, \la_3$ and seek to bound $\mathcal E(\la_1,\la_2,\la_3,\la_2)$. Note that this implies that $\la_3 \neq \la_1, \la_2$. Similarly to the previous subsection, suppose that
$$z=(z_1,z_2) \in (\mathcal A_{\lambda_1}+(a_{\lambda_2},\lambda_2a_{\lambda_2})\cdot\Delta(A))\cap (\mathcal A_{\lambda_3}+(a_{\lambda_2},\lambda_2a_{\lambda_2})\cdot\Delta(A) ).$$
Similar calculations show that we then have
$$0=\lambda_1a_1+\lambda_2a_{\lambda_2}a-\lambda_3a_3-\lambda_2a_{\lambda_2}b - \lambda_3(a_1+a_{\lambda_2}a-a_3-a_{\lambda_2}b),$$
for some $(a_1,a,a_3,b) \in A_{\la_1} \times A \times A_{\la_3}\times A$, and thus
\begin{equation}
0=(\la_1-\la_3)a_1 + a_{\la_2}(\la_2-\la_3)(a-b).
\label{CSsetupp}
\end{equation}

Let $T_2$ denote the number of solutions to \eqref{CSsetupp} such that $(a_1,a,b) \in A_{\la_1} \times A \times A$. We have shown that $\mathcal E( \la_1,\la_2,\la_3,\la_2) \leq T_2$. We can bound $T_2$ using the Cauchy-Schwarz inequality as in the case when $\la_2 \neq \la_4$ above. 
We obtain\footnote{Actually, Lemma \ref{Shk:assym} in a stronger forms says (see Corollary 32 of \cite{Sh3})  that either $|AB| \ll |A| |B|^{1/(4(k+4) 2^k)}$ or
	$|A\cap (A+x)| \ll |A| |B|^{-1/(4(k+4) 2^k)}$ for any nonzero $x$. Hence one can save a power of $\tau$ and bound 
	$T_2 \ll |A| |B|^{-1/(4(k+4) 2^k)} |A_{\lambda_1}| \ll |A| \tau^{1-2^{-213}}$. 
	We do not use these more accurate calculations in our proof.}
	
\begin{align*}
T_2&=\sum_{a \in A} \left | \left \{(b,a_1) \in A \times A_{\la_1} : a= b + \frac{\la_1-\la_3}{a_{\la_2}(\la_3-\la_2)}a_1 \right \} \right |
\\& \leq |A|^{1/2} \left( \sum_{x} \left | \left \{(b,a_1) \in A \times A_{\la_1} : x= b + \frac{\la_1-\la_3}{a_{\la_2}(\la_3-\la_2)}a_1 \right \} \right |^2 \right)^{1/2}
\\&=|A|^{1/2} E^+\left (A, \frac{\la_1-\la_3}{a_{\la_2}(\la_3-\la_2)} A_{\la_1} \right)^{1/2}.
\end{align*}
By inequality \eqref{keyenergy}, we then have
$$\mathcal E(\lambda_1,\lambda_2,\lambda_3,\lambda_2) \leq 12 |A|\tau^{1-2^{-214}}.$$
Combining this bound with \eqref{muu}, it follows that
\begin{equation}
\mu \geq \frac{M^2 |A|\tau}{2}  - 12M^3|A|\tau^{1-2^{-214}}.
\label{mu2}
\end{equation}
We impose another condition on $M$ so as to ensure that the first term remains dominant. To be precise, we impose the condition that
$$12M^3|A|\tau^{1-2^{-214}} \leq \frac{M^2|A| \tau }{4},$$
which simplifies to
\begin{equation}
M \leq \frac{ \tau^{2^{-214}}}{48}.
\label{Mcondd2}
\end{equation}
With this assumption, it follows that
\begin{equation}
 \mu \geq \frac{M^2|A|\tau}{4}.
\label{mufinal}
\end{equation}

\subsection{Choosing $M$ and concluding the proof}

We need to choose our integer parameter $M$ so that it satisfies both \eqref{Mcondd} and \eqref{Mcondd2}. We therefore finally fix
$$M:= \left \lfloor \frac{\tau^{2^{-215}}}{48}\right \rfloor.$$
We should check that this choice satisfies the condition that $2 \leq 2M \leq |S|$. The lower bound follows from the fact that $\tau \gg \frac{|A|^{1/25}}{\log^2|A|}$ by taking $n_0$ sufficiently large, while the upper bound follows from \eqref{Staubound}.

It then follows from \eqref{taubound2} and from taking $n_0$ sufficiently large that 
\begin{equation}
M \gg \tau^{2^{-215}}  \geq (|A|^{1/32})^{2^{-215}}=|A|^{2^{-220}}.
\label{Mfinal}
\end{equation}

Since this choice of $M$ is valid, we can now conclude the proof. The bound \eqref{mufinal} holds for all full clusters, and so we can sum over all $\lfloor |S|/2M \rfloor$ such clusters, also applying \eqref{Mfinal} and \eqref{Staubound}, to get
\begin{align*}
|AA+A|^2 \gg \frac{|S|}{M} M^2|A|\tau & =|S|M|A|\tau
\\& \gg (|S| \tau) |A|^{1+2^{-220}}
\\& \gg \frac{ |A|^{3+2^{-220}}}{\log^2|A|}
\\ & \geq  |A|^{3+2^{-221}}.
\end{align*}
We have thus finally proved that $|AA+A| \gg |A|^{3/2 + 2^{-222}}$.\qedsymbol

\subsection{A lower bound for the size of $AB+A$}

A small modification of the proof of Theorem \ref{thm:main} gives the following generalisation.

\begin{theorem} \label{thm:maingen} There is an integer $n_0$ and an absolute constant $c>0$ such that for any finite sets $A, B$ of positive reals with $|A|=|B| \geq n_0$,
$$|AB+A| \gg |A|^{\frac{3}{2}+c}.$$
\end{theorem}

The only significant change to the proof comes towards the beginning when dealing with the case when the ratio set $A/A$ is large. First apply Theorem \ref{thm:RSS} for $B$ to get large subsets $X_B$ and $Y_B$ with $E_+(X_B)^3E_*(Y_B) \lesssim |B|^{11}$. As in the proof of Theorem \ref{thm:main}, an application of Lemma \ref{bigratio} and the Ruzsa Triangle Inequality gives the desired bound unless $E_+(X_B) \geq |B|^{3-\epsilon}$. If this is the case then $E_*(Y_B) \leq |B|^{2+\epsilon'}$, and thus $|B/B| \geq |Y_B/Y_B| \geq |B|^{2-\epsilon'}$.

Now apply Theorem \ref{thm:RSS} for $A$ to get large subset $X_A$ and $Y_A$ satisfying the mixed energy bound. Then apply Lemma \ref{bigratio} again to get a lower bound 
$$|X_AB+X_AB| \gg |B/B|^{1/2}|A|(|A|^3/E_+(X_A))^{1/8}.$$ 
Combining this with the Ruzsa Triangle Inequality gives the desired result unless $E_+(X_A) \gg |A|^{3-\epsilon''}$. If this is the case then Theorem \ref{thm:RSS} implies that $E_*(Y_A) \ll |A|^{2+\epsilon'''}$.

We are done unless $E_*(Y_B)$ and $E_*(Y_A)$ are very small. Another application of the Cauchy-Schwarz inequality then gives the desired result because $|AB+A| \geq |AB| \geq |Y_AY_B| \gg |A|^{2-\epsilon'''}$.

We can therefore assume that $\tau \geq |A|^{\epsilon}$. The remainder of the proof of Theorem \ref{thm:maingen} is the essentially identical to that of Theorem \ref{thm:main}, with some obvious modifications. A slightly worse value for the constant $c$ is obtained.

\section{Remarks on the restriction to sets of positive reals}

When constructing an argument based on taking vector sums along pairs of lines through the origin, as was introduced to the sum-product problem in \cite{solymosi}, it is necessary to assume that the set one starts out with consists of only positive real numbers. This is typically not an important restriction, and the same results extend to arbitrary sets of real numbers. For example, if we know that \eqref{soly} holds for any finite set $A$ of positive reals then the bound also holds for an arbitrary finite $A \subset \mathbb R$. Indeed, since at least half of the elements of $A$ are either all positive or all negative, we can identity a subset $A' \subset A$ with $|A'|\gg |A|$ such that all elements of $A'$ have the same sign. If $A'$ consists of positive reals then we immediately obtain 
$$\max \{|A+A|,|AA|\} \geq \max \{|A'+A'|,|A'A'|\} \gg \frac{|A'|^{4/3}}{\log^{1/3}|A'|}\gg \frac{|A|^{4/3}}{\log^{1/3}|A|}.$$
On the other hand, if $A'$ is made up of negative values, we can simply apply \eqref{soly} for the set $-A'$ and obtain the same conclusion. So, in this case, the result generalises to arbitrary sets of reals, with only a slight weakening of the multiplicative constants.

Unfortunately, the situation is not as straightforward in the main result of this paper, and the condition that the set $A$ of consists of only positive reals is a more meaningful restriction. This is because the problem is not dilation invariant. Indeed, if we have a set $A$ of strictly negative real numbers and apply Theorem \ref{thm:main} to the set $-A$, we deduce that $|AA-A| \gg |A|^{\frac{3}{2}+c}$. Analogous results can be obtained for sets which contain a positive proportion of either positive or negative elements, but one cannot immediately extend Theorem \ref{thm:main} to arbitrary sets of real numbers.

These remarks are summarised in the form of the following theorem:
\begin{theorem} There is an integer $n_0$ and an absolute constant $c>0$ such that the following statement is true.
Let $A$ be a finite set of real numbers with $|A| \geq n_0$. If a positive proportion of elements of $A$ are positive, then
\begin{equation}
|AA+A| \gg |A|^{\frac{3}{2}+c}
\label{plus}
\end{equation}
If a positive proportion of elements of $A$ are negative, then
\begin{equation}
|AA-A| \gg |A|^{\frac{3}{2}+c}.
\label{minus}
\end{equation}
In particular, for any finite set of real numbers $A$ with $|A| \geq n_0$, at least one of \eqref{plus} and \eqref{minus} holds.
\end{theorem}

\section{Proof of Theorem \ref{thm:construction}}

We construct a set $A$ of integers such that $|A| =n$ and $|AA+m A|$ has cardinality
\[ O \left( n^2 (\log \log n )^{1- 2 \log 2 +o(1)} \right) . \]
for some integer $m$. This will complete the proof, since we can then take $B=m^{-1}A$ and check that $|BB+B|=|AA+mA|$.

We take a $q$ that is the product of primes up to a limit,
 \[ q = \prod_{p<y} p ,\]
where $y$ is taken so that $q^2<n$. By the prime number theorem we have $ y \sim (\log n)/2$. We put $m=q^2$.

Our set is defined via the "additive"\, function
 \[ f(x) = \sum_{p|x, p<y} 1 .\]
 This is periodic with period $q$. By our choice of the parameters in each interval of length $q$ its average is

  \[ \sum_{p<y} \frac{1}{p} = \log\log y + O(1) , \]
and its variance is
\[ \sum_{p<y}  \frac{1}{p} \left(1-  \frac{1}{p} \right)  = \log\log y + O(1) , \]
so by Chebyshev's inequality in each interval of length $q$ at least $q/2$ integers satisfy
 \[ f(x) > \log\log y - 2 \sqrt{\log\log y} .\]
Our set $A$ will be the collection of the first $n$ such numbers. By the above observation it will be
contained in at most $\lceil 2n/q \rceil $ blocks, hence $A\subset[1, 3n]$.

To describe the structure of $AA$ we introduce another additive function $g$, defined on primes $p<y$
by $g(p)=1$, $g(p^j)=2$ for $j>1$, and 0 on powers of greater primes. This has the property that
 \[ g(ab) \geq f(a) + f(b) \]
 for all integers $a,b$, consequently
 \[ g(x) > 2 \log\log y - 4 \sqrt{\log\log y} \]
for all $x\in AA$ .

This function $g$ is periodic with period $q^2$. To estimate the number of large values we calculate an exponential moment.
The average of $2^g$ in each block of length $q^2$ is exactly
 \[ \prod_{p<y} \left( 1 + \frac{1}{p}+\frac{2}{p^2} \right) \sim c \log y. \]
(To see this, observe that in every block $g$ is \emph{exactly} a sum of independent random variables, corresponding
to primes $p<y$ and assuming the values $0,1,2$ with probability $1-1/p, 1/p-1/p^2$ and $1/p^2$, respectively.
Hence $2^g$ is the product of variables assuming $1,2,4$ with the above probabilities.)
Consequently the proportion of residue classes modulo $q^2$ that intersect $AA$ can be estimated from above by
 \[  c  2^{   - 2 \log\log y + 4 \sqrt{\log\log y}} \log y  = (\log y)^{1- 2 \log 2 +o(1)}     = (\log \log n )^{1- 2 \log 2 +o(1)} .     \]
The same estimate holds for the  residue classes intersecting $AA+q^2 A$. As this set is contained in $[1, 10n^2]$,
its cardinality is
 \[ O \left( n^2 (\log \log n )^{1- 2 \log 2 +o(1)} \right) . \]

\section*{Acknowledgements} Oliver Roche-Newton was supported by the Austrian Science Fund (FWF): Project F5511-N26, which is part of the Special Research Program ``Quasi-Monte Carlo Methods: Theory and Applications" as well as by FWF Project P 30405-N32. Imre Z. Ruzsa was supported by ERC--AdG Grant No.321104  and Hungarian National Foundation for Scientific Research (OTKA), Grants No.109789 and NK104183. Chun-Yen Shen was supported by MOST, through grant 104-2628-M-002-015-MY4. 
Ilya Shkredov was supported in part by 
the Program of the Presidium of the Russian Academy of Sciences~01 ``Fundamental Mathematics and its Application" under grant PRAS-18-01.

We are grateful to Antal Balog, Brandon Hanson, Brendan Murphy, Friedrich Pillichshammer, Misha Rudnev, George Shakan and Dmitry Zhelezov for various helpful conversations and advice.

\end{document}